\documentclass[12pt]{amsart}

\usepackage{amsmath}
\usepackage{amscd}
\usepackage{amssymb}
\usepackage{amsfonts}
\usepackage{amsthm}
\usepackage{enumerate}
\usepackage{hyperref}
\usepackage{color}
\usepackage{graphicx}

\theoremstyle{plain}
\newtheorem{thm}{Theorem}[section]

\newtheorem{prop}[thm]{Proposition}
\newtheorem{lem}[thm]{Lemma}
\newtheorem{cor}[thm]{Corollary}

\newtheorem{ques}[thm]{Question}

\theoremstyle{definition}
\newtheorem{dfn}[thm]{Definition}

\newtheorem{rem}[thm]{Remark}

\newtheorem{dfns-rems}[thm]{Definitions and Remarks}
\newtheorem{notas-rems}[thm]{Notations and Remarks}
\newtheorem{exmps-rems}[thm]{Examples and Remarks}

\DeclareMathOperator{\min-match}{min-match}
\DeclareMathOperator{\ind-match}{ind-match}
\DeclareMathOperator{\ord-match}{ord-match}

\begin{document}


\title[Improved bounds for the regularity of powers of edge ideals]{Improved bounds for the regularity of powers of edge ideals of graphs}


\author[S. A. Seyed Fakhari]{S. A. Seyed Fakhari}

\address{S. A. Seyed Fakhari, School of Mathematics, Statistics and Computer Science,
College of Science, University of Tehran, Tehran, Iran, and Institute of Mathematics, Vietnam Academy of Science and Technology, 18 Hoang Quoc Viet, Hanoi, Vietnam.}

\email{aminfakhari@ut.ac.ir}

\author[S. Yassemi]{S. Yassemi}

\address{S. Yassemi, School of Mathematics, Statistics and Computer Science,
College of Science, University of Tehran, Tehran, Iran.}

\email{yassemi@ut.ac.ir}

\urladdr{http://math.ipm.ac.ir/$\sim$yassemi/}


\begin{abstract}
Let $G$ be a graph with edge ideal $I(G)$. We recall the notions of $\min-match_{\{K_2, C_5\}}(G)$ and $\ind-match_{\{K_2, C_5\}}(G)$ from \cite{sy}. We show that $${\rm reg}(I(G)^s)\leq 2s+\min-match_{\{K_2, C_5\}}(G)-1,$$for all $s\geq 1$, which implies that$${\rm reg}(I(G)^s)\leq 2s+\min-match(G)-1.$$Moreover, we show that$${\rm reg}(I(G)^s)\geq 2s+\ind-match_{\{K_2, C_5\}}(G)-2,$$and if $\ind-match_{\{K_2, C_5\}}(G)$ is an odd integer, then$${\rm reg}(I(G)^s)\geq 2s+\ind-match_{\{K_2, C_5\}}(G)-1.$$Furthermore, it is shown that$${\rm reg}(I(G)^s)\leq 2s+\ord-match(G)-1,$$where $\ord-match(G)$ denotes the ordered matching number of $G$. Finally, we construct infinitely many connected graphs which satisfy the following strict inequalities:$$2s+\ind-match(G)-1 < {\rm reg}(I(G)^s)< 2s+{\rm cochord}(G)-1.$$This gives a positive answer to a question asked in \cite{jns}.
\end{abstract}






\thanks{The research of the first author is partially funded by the Simons Foundation Grant Targeted for Institute of Mathematics, Vietnam Academy of Science and Technology.}


\maketitle


\section{Introduction} \label{sec1}

Let $\mathbb{K}$ be a field and $S = \mathbb{K}[x_1,\ldots,x_n]$  be the
polynomial ring in $n$ variables over $\mathbb{K}$. Suppose that $M$ is a graded $S$-module with minimal free resolution
$$0  \longrightarrow \cdots \longrightarrow  \bigoplus_{j}S(-j)^{\beta_{1,j}(M)} \longrightarrow \bigoplus_{j}S(-j)^{\beta_{0,j}(M)}   \longrightarrow  M \longrightarrow 0.$$
The Castelnuovo--Mumford regularity (or simply, regularity) of $M$, denote by ${\rm reg}(M)$, is defined as follows:
$${\rm reg}(M)=\max\{j-i|\ \beta_{i,j}(M)\neq0\}.$$
The regularity of $M$ is an important invariant in commutative algebra and algebraic geometry.

Cutkosky, Herzog, Trung, \cite{cht}, and independently Kodiyalam \cite{k1}, proved that for a homogenous ideal $I$ in a polynomial ring, reg$(I^s)$ is a linear function for $s\gg0$, i.e., there exist integers $a$, $b$, and $s_0$ such that $${\rm reg} (I^s)=as+b \ \ \ \ {\rm for \ all} \ s\geq s_0.$$
It is known that $a$ is bounded above by the maximum degree of elements in a minimal generating set of $I$. But a general bound for $b$ as well as $s_0$ is unknown.

There is a natural correspondence between quadratic squarefree monomial ideals of $S$ and finite simple graphs with $n$ vertices. To every simple graph $G$ with vertex set $V(G)=\{v_1, \ldots, v_n\}$ and edge set $E(G)$, we associate its {\it edge ideal} $I=I(G)$ defined by
$$I(G)=\big(x_ix_j: v_iv_j\in E(G)\big)\subseteq S.$$Computing and finding bounds for the regularity of edge ideals and their powers have been studied by a number of researchers (see for example \cite{ab},  \cite{abs}, \cite{b}, \cite{bbh1}, \cite{bbh}, \cite{bht}, \cite{dhs}, \cite{ha}, \cite{jns}, \cite{js}, \cite{k}, \cite{khm}, \cite{msy} and \cite{wo}).

Katzman \cite{k}, proved that for any graph $G$,
\[
\begin{array}{rl}
{\rm reg}(I(G))\geq \ind-match(G)+1,
\end{array} \tag{$\dagger$} \label{dag}
\]
where $\ind-match(G)$ denotes the induced matching number of $G$. Beyarslan, H${\rm \grave{a}}$ and Trung \cite{bht}, generalized Katzman's inequality by showing that$${\rm reg}(I(G)^s)\geq 2s+\ind-match(G)-1,$$for every integer $s\geq 1$. In 2014, Woodroofe, \cite[Theorem 1]{wo}, determined an upper bound for the regularity of edge ideals. Indeed, he proved ${\rm reg}(I(G))\leq {\rm cochord}(G)+1$, where ${\rm cochord}(G)$ denotes the co-chordal cover number of $G$. Alilooee, Banerjee, Beyarslan and H${\rm \grave{a}}$, \cite[Conjecture 7.11]{bbh1}, conjectured that for every graph $G$ and every integer $s\geq 1$, we have$${\rm reg}(I(G)^s)\leq 2s+{\rm cochord}(G)-1.$$This conjecture has been recently proved by Jayanthan and Selvaraja \cite[Theorem 4.4]{js}. Indeed, they prove the following stronger result to determine upper bounds for the regularity of powers of edge ideals.

\begin{lem} \label{hered}
{\rm (}\cite[Theorem 4.1]{js}{\rm )} Let $G$ be a graph and let $\mathcal{I}_G$ denote the family of induced subgraphs of $G$. Assume that $f:\mathcal{I}_G \rightarrow \mathbb{N}$ is a function which
satisfies the following properties.
\begin{itemize}
\item[(1)] For every graph $G\in \mathcal{I}_G$, we have ${\rm reg}(I(G))\leq f(G)+1$.
\item[(2)] If $H_1$ is an induced subgraph of $H_2$, then $f(H_1)\leq f(H_2)$.
\item[(3)] For any graph $H\in \mathcal{I}_G$ and every edge $e\in E(H)$, we have$$f(H-N_H[e])\leq f(H)-1.$$
\item[(4)] For every induced subgraph $H$ of $G$ with at least one edge, there exists a vertex $w\in V(H)$ such that$$f(G-N_G[w])\leq f(G)-1.$$
\end{itemize}
Then for every integer $s\geq 1$, we have$${\rm reg}(I(G)^s)\leq 2s+f(G)-1.$$
\end{lem}

The inequality ${\rm reg}(I(G)^s)\leq 2s+{\rm cochord}(G)-1$ is proved by the combination of the above Lemma with Lemma \ref{cochord} which was in fact proved in an earlier version of this paper. Because of this reason, we include Lemma \ref{cochord} also in this version and shortly explain how the above mentioned inequality follows from Lemmata \ref{hered} and \ref{cochord} (see Proposition \ref{main1}).

By \cite{wo}, we know that for every graph $G$,
\[
\begin{array}{rl}
{\rm reg}(I(G))\leq \min-match(G)+1,
\end{array}
\]
where $\min-match(G)$ denotes the minimum size of maximal matchings of $G$. This inequality was strengthened by the authors in \cite{sy}. In fact, in \cite{sy}, the authors introduced the notion of $\min-match_{\{K_2, C_5\}}(G)$, which is a lower bound for $\min-match(G)$ (see Definition \ref{newdef}). It is shown in \cite[Theorem 3.8]{sy} that
\[
\begin{array}{rl}
{\rm reg}(I(G))\leq \min-match_{\{K_2, C_5\}}(G)+1.
\end{array}
\]
The above inequality, suggests the following inequality, which is the first main result of this paper, Theorem \ref{regh}.
\[
\begin{array}{rl}
{\rm reg}(I(G)^s)\leq 2s+\min-match_{\{K_2, C_5\}}(G)-1.
\end{array}
\]
We mention that the proof of Theorem \ref{regh} is also based on Lemma \ref{hered}. We remark that recently Banerjee, Beyarslan and H${\rm \grave{a}}$ \cite[Theorem 3.4]{bbh} proved that for every graph $G$ and every integer $s\geq 1$, we have$${\rm reg}(I(G)^s)\leq 2s+{\rm match}(G)-1.$$Since $\min-match_{\{K_2, C_5\}}(G)$ is a lower bound for ${\rm match}(G)$, it follows that Theorem \ref{regh} is an improvement of \cite[Theorem 3.4]{bbh}.

There is another upper bound for the regularity of edge ideals, in terms of the ordered matching number of $G$ (see Definition \ref{om}). More precisely, let $G$ be a graph with ordered matching number $\ord-match(G)$. Constantinescu and Varbaro \cite[Remark 4.8]{cv} prove that ${\rm reg}(I(G))\leq \ord-match(G)+1$ (see also \cite[Corollary 2.5]{s4} for an alternative proof). As a generalization of this inequality, we prove in Theorem \ref{rego} that$${\rm reg}(I(G)^s)\leq 2s+\ord-match(G)-1,$$for every integer $s\geq 1$. Recently, Herzog and Hibi \cite[Theorem 1]{hh} proved that for every graph $G$ and every integer $s\geq 1$, we have $${\rm reg}(I(G)^s)\leq 2s+\alpha(G)-1,$$where $\alpha(G)$ denotes the independence number of $G$, which is the size of the largest independent subset of vertices of $G$. It is obvious from the definition of ordered matching number that this quantity is a lower bound for $\alpha(G)$. Thus, Theorem \ref{rego} is an improvement of \cite[Theorem 1]{hh}.

In Section \ref{sec4}, we determine a lower bound for the regularity of powers edge ideals. As mentioned above, Beyarslan, H${\rm \grave{a}}$ and Trung  proved that for every graph $G$ and every integer $s\geq 1$, we have$${\rm reg}(I(G)^s)\geq 2s+\ind-match(G)-1.$$In \cite{sy}, the authors introduced the notion of $\ind-match_{\{K_2, C_5\}}(G)$ which is an upper bound for $\ind-match(G)$ (see Definition \ref{newdef}). As an strengthen of inequality \ref{dag}, it was shown in \cite[Theorem 3.6]{sy} that for every graph $G$, we have $$\ind-match_{\{K_2, C_5\}}(G)+1\leq {\rm reg}(I(G)).$$This inequality suggests that$$2s+\ind-match_{\{K_2, C_5\}}(G)-1\leq {\rm reg}(I(G)^s) \ \ \ \ {\rm for \ all} \ s\geq 1.$$But the above inequality is not in general true, as the $5$-cycle graph $C_5$ shows. However, we prove in Theorem \ref{low} that for every graph $G$ and every integer $s\geq 1$, we have$$2s+\ind-match_{\{K_2, C_5\}}(G)-2\leq {\rm reg}(I(G)^s)$$and if $\ind-match_{\{K_2, C_5\}}(G)$ is an odd integer, then$$2s+\ind-match_{\{K_2, C_5\}}(G)-1\leq {\rm reg}(I(G)^s).$$

In Section \ref{sec5}, we investigate a question raised by Jayanthan, Narayanan and Selvaraja \cite[Question 5.8]{jns}. In fact, they asked wether there exists a graph $G$ with$$2s+ \ind-match(G)-1< {\rm reg}(I(G)^s)< 2s+{\rm cochord}(G)-1 \ \ \ \ {\rm for \ all} \ s\gg 0.$$Recently, Jayanthan and Selvaraja \cite{js} constructed a family of disconnected graphs which satisfy these inequalities for any $s\geq1$. In Section \ref{sec5}, we present infinitely many connected graphs for which the above strict inequalities hold true for every $s\geq 1$.


\section{Preliminaries} \label{sec2}

In this section, we provide the definitions and basic facts which will be used in the next sections.

Let $G$ be a simple graph with vertex set $V(G)=\big\{v_1, \ldots,
v_n\big\}$ and edge set $E(G)$. For a vertex $v_i$, the {\it neighbor set} of $v_i$ is $N_G(v_i)=\{v_j\mid v_iv_j\in E(G)\}$ and we set $N_G[v_i]=N_G(v_i)\cup \{v_i\}$ and call it the {\it closed neighborhood} of $v_i$. The cardinality of $N_G(v_i)$ is called the {\it degree} of $v_i$. For an edge $e=v_iv_j$ of $G$, we set $N_G[e]=N_G[v_i]\cup N_G[v_j]$. For every subset $U\subset V(G)$, the graph $G-U$ has vertex set $V(G-U)=V(G)\setminus U$ and edge set $E(G-U)=\{e\in E(G)\mid e\cap U=\emptyset\}$. A subgraph $H$ of $G$ is called {\it induced} provided that two vertices of $H$ are adjacent if and only if they are adjacent in $G$. The induced subgraph of $G$ on the vertex set $U\subseteq V(G)$ will be denoted by $G_U$. We recall that for a graph $G$, its {\it complementary graph}
$\overline{G}$ is the graph with $V(\overline{G})=V(G)$ and $E(\overline{G})$
consists of those $2$-element subsets $v_iv_j$ of $V(G)$ for which
$v_i,v_j\notin E(G)$. The complete graph with $n$ vertices will be denoted by $K_n$. A cycle graph with $n$ vertices is called an {\it $n$-cycle} graph an is denoted by $C_n$. A graph $G$ is called {\it chordal} if it has no induced cycle of length at least four. $G$ is said to be {\it co-chordal} if its complementary graph $\overline{G}$ is chordal. The minimum number of co-chordal subgraphs of $G$ which are needed to cover all edges of $G$ is called the {\it co-chordal cover number} of $G$ and is denoted by ${\rm cochord}(G)$. A subset $W$ of $V(G)$ is a {\it clique} of $G$, if every two distinct vertices of $W$ are adjacent in $G$. A vertex $v$ of $G$ is a {\it simplicial vertex} if $N_G(v)$ is a clique. It is well-known that every chordal graph has a simplicial vertex. The {\it girth} of $G$  is the length of the shortest cycle in $G$. A subset $A$ of $V(G)$ is called an {\it independent subset} of $G$ if there are no edges among the vertices of $A$. The cardinality of the largest independent subset of vertices of $G$ is called the {\it independence number} of $G$. Adding a {\it whisker} to $G$ at a vertex $v_i$ means adding a new vertex $u$ and the edge $uv_i$ to $G$. The graph which is obtained from $G$ by adding a whisker to all of its vertices is denoted by $W(G)$.

\begin{dfn}
A graph $G$ is called {\it vertex decomposable} if either it is an empty graph, or it has a vertex $v$ which satisfies the following conditions.
\begin{itemize}
\item[(i)] The graphs $G-v$ and $G-N_G[v]$ are vertex decomposable.
\item[(ii)] Every maximal independent subset of $G-v$ is a maximal independent set of $G$.
\end{itemize}
\end{dfn}

Let $G$ be a graph. A $5$-cycle of $G$ is said to be {\it basic} if it does not contain two adjacent vertices
of degree three or more in $G$. An edge of $G$ which is incident to a vertex of degree $1$ is called a {\it pendant} edge.
Let $C(G)$ denote the set of all vertices which belong to basic $5$-cycles and let $P(G)$
denote the set of vertices which are incident to pendant edges of $G$.

\begin{dfn} \label{defpc}
A graph $G$ is said to belong to the class $\mathcal{PC}$ if
\begin{itemize}
\item[(1)] $V(G)$ can be partitioned as $V(G)=P(G)\cup C(G)$, and
\item[(2)] the pendant edges form a perfect matching for the induced subgraph of $G$ on $P(G)$, and
\item[(3)] the vertices of basic $5$-cycles form a partition of $C(G)$.
\end{itemize}
\end{dfn}

By \cite[Theorem 2.4]{hmt}, a connected graph of girth at least $5$, belonging to the class $\mathcal{PC}$ is vertex decomposable.

Let $G$ be a graph. A subset $M\subseteq E(G)$ is a {\it matching} if $e\cap e'=\emptyset$, for every pair of edges $e, e'\in M$. The cardinality of the largest matching of $G$ is called the {\it matching number} of $G$ and is denoted by ${\rm match}(G)$. The minimum cardinality of the maximal matchings of $G$ is the {\it minimum matching number} of $G$ and is denoted by $\min-match(G)$. A matching $M$ of $G$ is an {\it induced matching} of $G$ if for every pair of edges $e, e'\in M$, there is no edge $f\in E(G)\setminus M$ with $f\subset e\cup e'$. An induced matching of size two is called a {\it gap}. It is clear that if $G$ has a gap, then its complementary graph $\overline{G}$ contains a $4$-cycle graph $C_4$ and hence, $G$ is not a co-chordal graph. The cardinality of the largest induced matching of $G$ is called the {\it induced  matching number} of $G$ and is denoted by $\ind-match(G)$.

We next recall the notions of $\ind-match_{\mathcal{H}}(G)$ and $\ind-match_{\mathcal{H}}(G)$ from \cite{sy}.

\begin{dfn} \label{newdef}
Let $G$ be a graph with at least one edge and let $\mathcal{H}$ be a collection of connected graphs with $K_2\in \mathcal{H}$. We say that a subgraph $H$ of $G$, is an {\it $\mathcal{H}$--subgraph} if every connected component of $H$ belongs to $\mathcal{H}$. If moreover $H$ is an induced subgraph of $G$, then we say that it is an {\it induced $\mathcal{H}$--subgraph} of $G$. Since $K_2\in \mathcal{H}$, every graph with at least one edge has an induced $\mathcal{H}$--subgraph. An $\mathcal{H}$--subgraph $H$ of $G$ is called {\it maximal} if $G\setminus V(H)$ has no $\mathcal{H}$--subgraph. We set$$\ind-match_{\mathcal{H}}(G):={\rm max} \big\{{\rm match}(H)\mid H \ {\rm is \ an \ induced} \ \mathcal{H}\textendash {\rm subgraph \ of} \ G \big\},$$and$$\min-match_{\mathcal{H}}(G):={\rm min} \big\{{\rm match}(H)\mid H \ {\rm is \ a \ maximal} \ \mathcal{H}\textendash {\rm subgraph \ of} \ G \big\},$$and call them the {\it induced $\mathcal{H}$--matching number} and the {\it minimum $\mathcal{H}$--matching number} of $G$, respectively. We set $\ind-match_{\mathcal{H}}(G)=\min-match_{\mathcal{H}}(G)=0$, when $G$ has no edge.
\end{dfn}

Of particular interest is the case $\mathcal{H}=\{K_2, C_5\}$. Indeed, we know from \cite[Corollary 3.9]{sy} that for every graph $G$ with edge ideal $I(G)$, we have$$\ind-match_{\{K_2, C_5\}}(G)+1\leq{\rm reg}(I(G))\leq \min-match_{\{K_2, C_5\}}(G)+1.$$Note that for every graph $G$, the quantity $\ind-match_{\{K_2, C_5\}}(G)$ is an upper bound for $\ind-match(G)$ and $\min-match_{\{K_2, C_5\}}(G)$ is a lower bound for $\min-match(G)$.

We close this section by the definition of ordered matching number. 

\begin{dfn} \label{om}
Let $G$ be a graph, and let $M=\big\{\{a_i,b_i\}\mid 1\leq i\leq r\big\}$ be a
nonempty matching of $G$. We say that $M$ is an {\it ordered matching} of
$G$ if the following hold:
\begin{itemize}
\item[(1)] $A:=\{a_1,\ldots,a_r\} \subseteq V(G)$ is a set of
    independent vertices of $G$; and

\item[(2)] $\{a_i, b_j\}\in E(G)$ implies that $i\leq j$.
\end{itemize}
The {\it ordered matching number} of $G$, denoted by $\ord-match(G)$, is
defined to be $$\ord-match(G)=\max\{|M|\mid M\subseteq E(G)\ {\rm is\ an\
ordered\ matching\ of} \ G\}.$$
\end{dfn}


\section{Upper Bounds} \label{sec3}

In this section, we determine two upper bounds for the regularity of powers of edge ideals, Theorems \ref{regh} and \ref{rego}. Before focusing on our main results, we first prove Lemma \ref{cochord}, which is a part of the proof of Jayanthan and Selvaraja \cite[Theorem 4.4]{js} for the inequality$${\rm reg}(I(G)^s)\leq 2s+{\rm cochord}(G)-1.$$Next, in Proposition \ref{main1}, we shortly explain the proof of the above inequality, using Lemmata \ref{hered} and \ref{cochord}.

\begin{lem} \label{cochord}
Let $G$ be a graph with at least one edge. Then there is a vertex $w\in V(G)$ such that$${\rm cochord}(G-N_G[w])\leq {\rm cochord}(G)-1.$$
\end{lem}

\begin{proof}
Assume that ${\rm cochord}(G)=t$ and let $G_1, \ldots, G_t$ be the co-chordal subgraphs of $G$ with $E(G)=\bigcup_{i=1}^tE(G_i)$. As $G$ has at least one edge, we conclude that $t\geq 1$. Suppose that $w$ is simplicial vertex of $\overline{G_1}$. Assume that $N_{\overline{G_1}}(w)=\{w_1, \ldots, w_s\}$. Since $w_1, \ldots, w_s$ form a clique in $\overline{G_1}$, it follows that they are independent vertices of $G_1$. Notice that $V(G_1-N_{G_1}[w])=\{w_1, \ldots, w_s\}$, which means that $G_1-N_{G_1}[w]$ consists of isolated vertices. For every $1\leq i\leq t$, set $W_i=V(G_i)\cap N_G[w]$. Since $G_1-W_1$ is a subgraph of $G_1-N_{G_1}[w]$, we conclude that $G_1-W_1$ has no edge. Thus,$$E(G-N_G[w])=\bigcup_{i=1}^tE(G_i-W_i)=\bigcup_{i=2}^tE(G_i-W_i).$$Note that for every integer $i$ with $2\leq i\leq t$, the graph $G_i-W_i$ is a co-chordal graph. Hence, ${\rm cochord}(G-N_G[w])\leq t-1$.
\end{proof}

As we mentioned above, Lemma \ref{cochord} together with Lemma \ref{hered} provides an upper bound for the regularity of powers of edge ideals in terms of the cochordal cover number.

\begin{prop} \label{main1}
{\rm (}\cite[Theorem 4.4]{js}{\rm )} For every graph $G$ and every integer $s\geq 1$, we have$${\rm reg}(I(G)^s)\leq 2s+{\rm cochord}(G)-1.$$
\end{prop}

\begin{proof}
For every graph $G$, we set $f(G)={\rm cochord}(G)$. We know from \cite[Theorem 1]{wo} that ${\rm reg}(I(G))\leq f(G)+1$. On the other hand, it follows from Lemma \ref{cochord} that for every graph $G$, there exists a vertex $w\in V(G)$ with $f(G-N_G[w])\leq f(G)-1$. Obviously, for any induced subgraph $H$ of $G$ we have $f(H)\leq f(G)$. Also, for any edge $e$ of $G$, it is clear that the disjoint union of $e$ and $G-N_G[e]$ is an induced subgraph of $G$. This implies that $f(G-N_G[e])\leq f(G)-1$. Hence, Lemma \ref{hered} implies that$${\rm reg}(I(G)^s)\leq 2s+f(G)-1=2s+{\rm cochord}(G)-1.$$
\end{proof}

Now, we start the proof of the first main result of this paper, Theorem \ref{regh}, which states that for every graph $G$ and every integer $s\geq 1$, the inequality 
\[
\begin{array}{rl}
{\rm reg}(I(G)^s)\leq 2s+\min-match_{\{K_2, C_5\}}(G)-1
\end{array}
\]
holds. The proof of the above inequality is based on Lemma \ref{hered}. In the following three lemmas, we verify the assumptions of Lemma \ref{hered}.

\begin{lem} \label{minmatch}
Let $G$ be a graph with at least one edge. Then there is a vertex $w\in V(G)$ such that$$\min-match(G-N_G[w])\leq \min-match(G)-1.$$
\end{lem}

\begin{proof}
Suppose that $\min-match(G)=t$ and consider a maximal matching $\{e_1, \ldots, e_t\}$ of $G$. Let $w$ be a vertex of $e_t$. Without loss of generality, we assume there exist nonnegative integers $p$ and $q$ such that
\begin{itemize}
\item[(i)] for every integer $i$ with $1\leq i\leq p$, the edge $e_i$ is not incident to any vertex in $N_G[w]$;
\item[(ii)] for every integer $i$ with $p+1\leq i\leq p+q$, the edge $e_i$ is incident to exactly one vertex in $N_G[w]$;
\item[(iii)] for every integer $i$ with $p+q+1\leq i\leq t$, the both vertices of $e_i$ belong to $N_G[w]$.
\end{itemize}
As the vertices of $e_t$ belong to $N_G[w]$, we conclude that $p+q<t$.

For every integer $i$ with $p+1\leq i\leq p+q$, let $v^i$ be the vertex of $e_i$ which does not belong to $N_G[w]$. Assume that $U=\{u_1, \ldots, u_m\}$ is the set of vertices of $G$ which are not incident to $e_1, \ldots, e_t$ and set$$U'=\big(V(G-N_G[w])\cap U\big)\cup\{v^{p+1}, \ldots, v^{p+q}\}.$$As $\{e_1, \ldots, e_t\}$ is a maximal matching of $G$, we conclude that $U$ is an independent subset of vertices of $G$. Thus, every edge of the induced subgraph $(G-N_G[w])_{U'}$ is adjacent to at least one of the vertices $v^{p+1}, \ldots, v^{p+q}$. This means that$${\rm match}((G-N_G[w])_{U'})\leq q.$$ Let $e'_1, \ldots, e'_r$ be a maximal matching of $(G-N_G[w])_{U'})$. In particular, $r\leq q$. Note that $\{e_1, \ldots, e_p, e'_1, \ldots, e'_r\}$ is a maximal matching of $G-N_G[w]$. Thus,$$\min-match(G-N_G[w])\leq p+r\leq p+q<t,$$as required.
\end{proof}

\begin{lem} \label{minh}
Let $G$ be a graph with at least one edge. Then there is a vertex $w\in V(G)$ such that$$\min-match_{\{K_2, C_5\}}(G-N_G[w])\leq \min-match_{\{K_2, C_5\}}(G)-1.$$
\end{lem}

\begin{proof}
Assume that $\min-match_{\{K_2, C_5\}}(G)=t$ and let $H$ be a maximal $\{K_2, C_5\}$-subgraph of $G$ with ${\rm match}(H)=t$. Suppose that $\{e_1, \ldots, e_m, G_1, \ldots, G_s\}$ is the set of connected components of $H$, where $e_1, \ldots, e_m$ are isomorphic to $K_2$ and $G_1, \ldots, G_s$ are $5$-cycles. Thus, $m+2s=t$. if $s=0$, then$$\min-match_{\{K_2, C_5\}}(G)=\min-match(G)$$and it follows from Lemma \ref{minmatch} that there exists a vertex $w\in V(G)$ with $$\min-match_{\{K_2, C_5\}}(G-N_G[w])\leq \min-match(G-N_G[w])$$$$\leq \min-match(G)-1=\min-match_{\{K_2, C_5\}}(G)-1.$$Thus, assume that $s\geq 1$. Let $w$ be a vertex of $G_s$. Without loss of generality, we suppose there exist nonnegative integers $p_1, q_1, q_2$ such that
\begin{itemize}
\item[(i)] for every integer $i$ with $1\leq i\leq p_1$, the cycle $G_i$ has no vertex in $N_G[w]$;
\item[(ii)] for every integer $i$ with $p_1+1\leq i\leq s$, the cycle $G_i$ has at least one vertex in $N_G[w]$;
\item[(iii)] for every integer $i$ with $1\leq i\leq q_1$, the edge $e_i$ is not incident to any vertex in $N_G[w]$;
\item[(iv)] for every integer $i$ with $q_1+1\leq i\leq q_1+q_2$, the edge $e_i$ is incident to exactly one vertex in $N_G[w]$;
\item[(v)] for every integer $i$ with $q_1+q_2+1\leq i\leq m$, the both vertices of $e_i$ belong to $N_G[w]$.
\end{itemize}
As $w$ is a vertex of $G_s$, we conclude that $p_1<s$.

For every integer $i$ with $p_1+1\leq i\leq s$, set $W_i=V(G_i)\setminus N_G[w]$ and consider the graph $H_i=(G_i)_{W_i}$ (the induced subgraph of $G_i$ on $W_i$). Let $M_i$ be a matching of $H_i$ of size ${\rm match}(H_i)$ and let $L_i$ be the set of vertices of $H_i$ which are not covered by any edge of $M_i$. Notice that $H_i$ has at most four vertices and it is easy to check that for every integer $i$ with $p_1+1\leq i\leq s-1$, we have $|M_i|+|L_i|\leq 2$. On the other hand, recall that $w$ is a vertex of $G_s$ and hence, $|M_s|+|L_s|\leq 1$.

For every integer $i$ with $q_1+1\leq i\leq q_1+q_2$, let $v^i$ be the vertex of $e_i$ which does not belong to $N_G[w]$. Suppose that $U=\{u_1, \ldots, u_{\ell}\}$ is the set of vertices of $G$ which do not belong to $V(H)$ and set$$U'=\big(V(G-N_G[w])\cap U\big)\cup \big(\bigcup_{i=p_1+1}^sL_i\big)\cup\big\{v^{q_1+1}, \ldots, v^{q_1+q_2}\big\}.$$As $H$ is a maximal $\{K_2, C_5\}$-subgraph of $G$, we conclude that $U$ is an independent subset of vertices of $G$. Thus, every edge of the induced subgraph $(G-N_G[w])_{U'}$ is adjacent to at at least one vertex in the set$$\big(\bigcup_{i=p_1+1}^sL_i\big)\cup\big\{v^{q_1+1}, \ldots, v^{q_1+q_2}\big\}.$$This means that$${\rm match}((G-N_G[w])_{U'})\leq q_2+\sum_{i=p_1+1}^s|L_i|.$$Let $e'_1, \ldots, e'_r$ be a maximal matching of $(G-N_G[w])_{U'})$. In particular, $$r\leq q_2+\sum_{i=p_1+1}^s|L_i|.$$Note that the edges of the set$$\{e_1, \ldots, e_{q_1}, e'_1, \ldots, e'_r\}\cup\big(\bigcup_{i=p_1+1}^sM_i\big),$$ together with the $5$-cycles $G_1, \ldots, G_{p_1}$ is a maximal $\{K_2, C_5\}$-subgraph of $G-N_G[w]$. Thus,
\begin{align*}
& \min-match_{\{K_2, C_5\}}(G-N_G[w]) \leq q_1+r+2p_1+\sum_{i=p_1+1}^s|M_i|\\ & \leq q_1+q_2+2p_1 +\sum_{i=p_1+1}^s|L_i|+\sum_{i=p_1+1}^s|M_i|\\ &  =q_1+q_2+2p_1+|L_s|+|M_s|+\sum_{i=p_1+1}^{s-1}|L_i|+\sum_{i=p_1+1}^{s-1}|M_i|\\ & \leq q_1+q_2+2p_1+1+2(s-p_1-1) =q_1+q_2+2s-1\\ & \leq m+2s-1=t-1,
\end{align*}
as required.
\end{proof}

\begin{lem} \label{minh2}
For every graph $G$ and any vertex $w\in V(G)$, we have$$\min-match_{\{K_2, C_5\}}(G-w)\leq \min-match_{\{K_2, C_5\}}(G).$$
\end{lem}

\begin{proof}
Assume that $\min-match_{\{K_2, C_5\}}(G)=t$ and let $H$ be a maximal $\{K_2, C_5\}$-subgraph of $G$ with ${\rm match}(H)=t$. Suppose that $\{e_1, \ldots, e_m, G_1, \ldots, G_s\}$ is the set of connected components of $H$, where $e_1, \ldots, e_m$ are isomorphic to $K_2$ and $G_1, \ldots, G_s$ are $5$-cycles. In particular, $m+2s=t$. We consider the following cases.\\

{\bf Case1.} If $w\notin V(H)$, then $H$ is a maximal $\{K_2, C_5\}$-subgraph of $G-w$ and hence, $\min-match_{\{K_2, C_5\}}(G-w)\leq t$.\\

{\bf Case 2.} Suppose that $w$ is a vertex of $G_i$, for some integer $i$ with $1\leq i\leq s$. Without loss of generality, we may assume that $i=1$. Then $G_1-w$ has a matching $e_1', e_2'$ of size $2$. Then the edges $e_1, \ldots, e_m, e_1', e_2'$ together with the cycles $G_2, \ldots, G_s$ form a maximal $\{K_2, C_5\}$-subgraph of $G-w$, with matching number $m+2+2(s-1)=t$. Thus, $\min-match_{\{K_2, C_5\}}(G-w)\leq t$.\\

{\bf Case 3.} Suppose that $w$ is a vertex of $e_i$, for some integer $i$ with $1\leq i\leq m$. Without loss of generality, we may assume that $i=1$. Let $v$ be the other vertex of $e_1$. Set $U=V(G)\setminus V(H)$ and $U'=U\cup \{v\}$. As $H$ is a maximal $\{K_2, C_5\}$-subgraph of $G$, we conclude that $U$ is an independent subset of vertices of $G$. In particular, every edge of $G_{U'}$ is adjacent to $v$. Hence, ${\rm match}(G_{U'})\leq 1$. Suppose that $M$ is a maximal matching of $G_{U'}$. In particular, $|M|\leq 1$. Note that the edges of the set $\{e_2, \ldots, e_m\}\cup M$ together with the cycles $G_1, \ldots, G_s$ form a maximal $\{K_2, C_5\}$-subgraph of $G-w$, with matching number $\leq m+2s=t$. Thus, $\min-match_{\{K_2, C_5\}}(G-w)\leq t$.
\end{proof}

We are now ready to prove the first main result of this paper.

\begin{thm} \label{regh}
For every graph $G$ and for every integer $s\geq 1$, we have$${\rm reg}(I(G)^s)\leq 2s+\min-match_{\{K_2, C_5\}}(G)-1.$$
\end{thm}

\begin{proof}
For any graph $G$, set $f(G)=\min-match_{\{K_2, C_5\}}(G)$. We know from \cite[Theorem 3.8]{sy} that ${\rm reg}(I(G))\leq f(G)+1$. It follows from Lemma \ref{minh} that for every $G$ with at least on edge, there exists a vertex $w\in V(G)$ with $f(G-N_G[w])\leq f(G)-1$. We also know by Lemma \ref{minh2} that for any induced subgraph $H$ of $G$, the inequality $f(H)\leq f(G)$ holds. Let $e$ be an edge of $G$ and let $L$ be the disjoin union of $G-N_G[e]$ and $e$. Then $L$ is an induced subgraph of $G$. Thus,$$f(G-N_G[e])\leq f(L)-1\leq f(G)-1,$$where the first inequality follows from the definition of $f$ and the second inequality follows Lemma \ref{minh2}. Hence, Lemma \ref{hered} implies that$${\rm reg}(I(G)^s)\leq 2s+f(G)-1=2s+\min-match_{\{K_2, C_5\}}(G)-1.$$
\end{proof}

Banerjee, Beyarslan and H${\rm \grave{a}}$, \cite[Theorem 3.4]{bbh}, prove that for every graph $G$ and every integer $s\geq 1$,$${\rm reg}(I(G)^s)\leq 2s+{\rm match}(G)-1.$$The following corollary is an immediate consequence of Theorem \ref{regh} and improves \cite[Theorem 3.4]{bbh}.

\begin{cor} \label{minm}
For every graph $G$ and for every integer $s\geq 1$, we have$${\rm reg}(I(G)^s)\leq 2s+\min-match(G)-1.$$
\end{cor}

\begin{proof}
The assertion follows from Theorem \ref{regh} and the inequality$$\min-match_{\{K_2, C_5\}}(G)\leq \min-match(G).$$
\end{proof}

\begin{cor}
Let $G$ be a graph with $\ind-match(G)=\min-match(G)$. Then for every integer $s\geq 1$, we have$${\rm reg}(I(G)^s)=2s+\ind-match(G)-1$$
\end{cor}

\begin{proof}
We know from \cite[Theorem 4.5]{bht} and Corollary \ref{minm} that for every integer $s\geq 1$,$$2s+\ind-match(G)-1\leq {\rm reg}(I(G)^s)\leq 2s+\min-match(G)-1.$$The assertion now follows from the hypothesis.
\end{proof}

We recall that a characterization of graphs which satisfy the equality $\ind-match(G)=\min-match(G)$ was obtained in \cite[Theorem 2.3]{hhkt}.

As the final result of this section, we determine an upper bound for the regularity of powers of edge ideals, in terms of the ordered matching number. It improves the result of Herzog and Hibi \cite[Theorem 1]{hh}.

\begin{thm} \label{rego}
For every graph $G$ and for every integer $s\geq 1$, we have$${\rm reg}(I(G)^s)\leq 2s+\ord-match(G)-1.$$
\end{thm}

\begin{proof}
For any graph $G$, set $f(G)=\ord-match(G)$. We know from \cite[Remark 4.8]{cv} (see also \cite[Corollary 2.5]{s4}) that ${\rm reg}(I(G))\leq f(G)+1$. It follows from \cite[Lemma 2.1]{s4} that for every vertex $w\in V(G)$ we have $f(G-N_G[w])\leq f(G)-1$. It is clear that for any induced subgraph $H$ of $G$, the inequality $f(H)\leq f(G)$ holds. Let $e=xy$ be an edge of $G$. Then $G-N_G[x]$ is an induced subgraph of $G-N_G[e]$. Therefore,$$f(G-N_G[e])\leq f(G-N_G[x])\leq f(G)-1,$$where the last inequality follows from \cite[Lemma 2.1]{s4}. Hence, Lemma \ref{hered} implies that$${\rm reg}(I(G)^s)\leq 2s+f(G)-1= 2s+\ord-match(G)-1.$$
\end{proof}


\section{A lower bound} \label{sec4}

In this section, we determine a lower bound for the regularity of powers of edge ideals. It was shown in \cite[Theorem 3.6]{sy} that for every graph $G$, we have $$\ind-match_{\{K_2, C_5\}}(G)+1\leq {\rm reg}(I(G)).$$Based on this inequality, one may guess that for every graph $G$ and every integer $s\geq 1$, we have$$2s+\ind-match_{\{K_2, C_5\}}(G)-1\leq {\rm reg}(I(G)^s).$$But, as we mentioned in the introduction, this inequality is not true in general. However, we have the following result.

\begin{thm} \label{low}
For every graph $G$ and for every integer $s\geq 1$, we have$${\rm max}\{2s+\ind-match(G)-1, 2s+\ind-match_{\{K_2, C_5\}}(G)-2\}\leq {\rm reg}(I(G)^s).$$If moreover, $\ind-match_{\{K_2, C_5\}}(G)$ is an odd integer, then$$2s+\ind-match_{\{K_2, C_5\}}(G)-1\leq {\rm reg}(I(G)^s) \ \ \ \ \ \ \ {\rm for \ all} \ s\geq 1.$$
\end{thm}

\begin{proof}
The inequality $2s+\ind-match(G)-1\leq {\rm reg}(I(G)^s)$ is known by \cite[Theorem 4.5]{bht}. Thus, we only need to prove that
\[
\begin{array}{rl}
{\rm reg}(I(G)^s)\geq 2s+\ind-match_{\{K_2, C_5\}}(G)-2 \ \ \ \ \ \ \ {\rm for \ all} \ s\geq 1
\end{array} \tag{$\ast$} \label{astast}
\]
For $s=1$, the inequality (\ref{astast}) follows from \cite[Corollary 3.9]{sy}. Hence, suppose that $s\geq 2$. Assume that $\ind-match_{\{K_2, C_5\}}(G)=t$ and let $H$ be an induced $\{K_2, C_5\}$--subgraph of $G$ with ${\rm match}(H)=t$. By \cite[Corollary 4.3]{bht}, it is enough to prove that
\[
\begin{array}{rl}
{\rm reg}(I(H)^s)\geq 2s+t-2 \ \ \ \ \ \ \ {\rm for \ all} \ s\geq 2
\end{array} \tag{$\ast\ast$} \label{ast}
\]
Suppose that $\{e_1, \ldots, e_m, G_1, \ldots, G_k\}$ is the set of connected components of $H$, where $e_1, \ldots, e_m$ are isomorphic to $K_2$ and $G_1, \ldots, G_k$ are $5$-cycles. We use induction on $k$. If $k=0$, then the inequality (\ref{ast}) follows from \cite[Lemma 4.4]{bht}. Thus, assume that $k\geq 1$. Let $H'$ be the graph with connected components $e_1, \ldots, e_m, G_1, \ldots, G_{k-1}$. Then $I(H)=I(H')+I(G_k)$. If $I(H')=0$ (i.e., $m=0$ and $k=1$), then $I(H)=I(G_k)$ and the inequality (\ref{ast}) follows from \cite[Theorem 5.2]{bht}. Hence, suppose that $I(H')\neq 0$. In this case, \cite[Theorem 1.1]{nv} implies that$${\rm reg}(I(H)^s)\geq {\rm reg}(I(H')^s)+{\rm reg}(I(G_k))-1.$$As $G_k$ is a $5$-cycle, we know that ${\rm reg}(I(G_k))=3$. On the other hand, it follows from the induction hypothesis that$${\rm reg}(I(H')^s)\geq 2s+\ind-match_{\{K_2, C_5\}}(H')-2=2s+t-4.$$Therefore,$${\rm reg}(I(H)^s)\geq {\rm reg}(I(H')^s)+{\rm reg}(I(G_k))-1=2s+t-4+3-1=2s+t-2.$$For the last part of theorem, again notice that the case $s=1$ follows from \cite[Corollary 3.9]{sy}. For $s\geq 2$, we use the similar argument (and the same notations) as above. As $t$ is an odd integer, it follows that $m\neq 0$ and hence, $I(H')\neq 0$. On the other hand, $\ind-match_{\{K_2, C_5\}}(H')=t-2$ is an odd integer. Therefore, the induction hypothesis implies that$${\rm reg}(I(H')^s)\geq 2s+\ind-match_{\{K_2, C_5\}}(H')-1=2s+t-3.$$Again, \cite[Theorem 1.1]{nv} implies that$${\rm reg}(I(H)^s)\geq {\rm reg}(I(H')^s)+{\rm reg}(I(G_k))-1\geq 2s+t-3+3-1=2s+t-1.$$
\end{proof}

\begin{rem}
The proof of Theorem \ref{low} shows that if $G$ has an induced $\{K_2, C_5\}$--subgraph $H$ with $\ind-match_{\{K_2, C_5\}}(G)={\rm match}(H)$ such that at least one connected component of $H$ is isomorphic to $K_2$, then$$2s+\ind-match_{\{K_2, C_5\}}(G)-1\leq {\rm reg}(I(G)^s) \ \ \ \ \ \ \ {\rm for \ all} \ s\geq 1.$$
\end{rem}


\section{An Example} \label{sec5}

In this section, we investigate the following question asked by Jayanthan, Narayanan and Selvaraja.

\begin{ques} \label{ques}
{\rm (}\cite[Question 5.8]{jns}{\rm )} Does there exist any graph $G$ which satisfies the inequalities$$2s+ \ind-match(G)-1< {\rm reg}(I(G)^s)< 2s+{\rm cochord}(G)-1$$for every integer $s\gg 0$?
\end{ques}

Recently, Jayanthan and Selvaraja \cite{js} gave a positive answer to this question by constructing a family of disconnected graphs for which the above inequalities hold. In this section, we show the answer of Question \ref{ques} is again positive, if we restrict ourselves to the category of connected graphs. In other words, we present infinitely many
connected graphs which satisfy the strict inequalities of Question \ref{ques}, for every $s\geq 1$.

For every integer $n\geq 1$, let $H_n$ be the graph with vertex set$$V(H_n)=\bigcup_{i=1}^n\{v_1^i, v_2^i, v_3^i, v_4^i, v_5^i\}$$and edge set$$E(H_n)=\bigcup_{i=1}^n\{v_1^iv_2^i, v_2^iv_3^i, v_3^iv_4^i, v_4^iv_5^i, v_1^iv_5^i\}\cup\{v_3^iv_1^{i+1} | 1\leq i\leq n-1\}.$$The graph $H_3$ is shown in Figure \ref{trian}. Assume that $W(H_n)$ is the graph obtained from $H_n$ by attaching a whisker to every vertex of $H_n$. As $H_n$ has no triangle, we conclude that among any three vertices of $H_n$, at least two of them are independent. This means that among any three whiskers of $W(H_n)$, at least two of them form a gap. Hence, any co-chordal subgraph of $W(H_n)$ contains at most two whiskers. This implies that$${\rm cochord}(W(H_n))\geq \frac{|V(H_n)|}{2}=\frac{5n}{2}.$$We know from \cite{v1} that $W(H_n)$ is a Cohen--Macaulay graph. By \cite[Theorem 4.3]{sy} and \cite[Theorem 13]{bv},
\[
\begin{array}{rl}
\ind-match_{\{K_2, C_5\}}(W(H_n))={\rm reg}(S/I(W(H_n)))=\ind-match(W(H_n)).
\end{array} \tag{$\ddagger$} \label{ddag}
\]
Since the independence number of a $5$-cycle is two, it follows that the independence number of $H_n$ is at most $2n$ and hence,$$\ind-match(W(H_n))\leq 2n.$$
\begin{figure}[h!]
\centering
\includegraphics{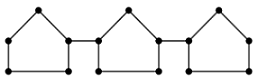}
\vspace{-0.1cm}
\caption{The graph $H_3$} \label{trian}
\end{figure}

Let $H$ be the graph shown Figure \ref{graph}. Note that $\ind-match_{\{K_2, C_5\}}(H)=6$ and $\ind-match(H)=4$.
\begin{figure}[h!]
\centering
\includegraphics{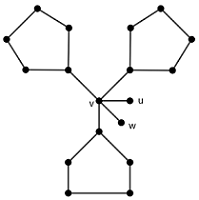}
\vspace{-0.1cm}
\caption{The graph $H$} \label{graph}
\end{figure}

Assume that $x$ is an arbitrary vertex of $H_n$ and suppose that $G_n$ is the graph obtained from $H$ and $W(H_n)$ by identifying the vertices $u$ and $x$, i.e., $G_n$ is the graph with vertex set$$V(G_n)=V(H\setminus u)\cup V(W(H_n)\setminus x)\cup \{z\},$$(where $z$ is a new vertex) and its edge set is defined as$$E(G_n)=E(H\setminus u)\cup E(W(H_n)\setminus x)\cup \{zy\mid y\in N_H(u)\cup N_{W(H_n)}(x)\}.$$As $W(H_n)$ is an induced subgraph of $G_n$, we conclude that$${\rm cochord}(G_n)\geq {\rm cochord}(W(H_n))\geq \frac{5n}{2}.$$On the other hand, it is clear that$$\ind-match(G_n)\leq \ind-match(W(H_n))+\ind-match(H)=\ind-match(W(H_n))+4.$$

To compute $\ind-match_{\{K_2, C_5\}}(G_n)$, let $L$ be an induced $\{K_2, C_5\}$--subgraph of $W(H_n)$. Then the union of $L$ and the three $5$-cycles of $H$ forms an induced $\{K_2, C_5\}$--subgraph of $G_n$. Thus,
\begin{align*}
& \ind-match_{\{K_2, C_5\}}(W(H_n))+6\leq \ind-match_{\{K_2, C_5\}}(G_n)\\ & \leq \ind-match_{\{K_2, C_5\}}(W(H_n))+\ind-match_{\{K_2, C_5\}}(H)\\ & =\ind-match_{\{K_2, C_5\}}(W(H_n))+6,
\end{align*}
Therefore,
\begin{align*}
& \ind-match_{\{K_2, C_5\}}(G_n)=\ind-match_{\{K_2, C_5\}}(W(H_n))+6\\ & =\ind-match(W(H_n))+6\leq 2n+6,
\end{align*}
where the second equality follows from the equalities \ref{ddag}.

Setting$$P(G_n)=V(W(H_n)\setminus x)\cup\{z, v, w\}$$ and$$C(G_n)=V(H)\setminus \{u, v, w\},$$we see that the graph $G_n$ belongs to the class $\mathcal{PC}$ and hence, by \cite[Theorem 2.4]{hmt}, it is a vertex decomposable graph. Thus, for any $n\geq 13$ and every $s\geq 1$, we have
\begin{align*}
& 2s+\ind-match(G_n)-1\leq 2s+\ind-match(W(H_n))+4-1\\ & =2s+\ind-match_{\{K_2, C_5\}}(W(H_n))+3 < 2s+\ind-match_{\{K_2, C_5\}}(W(H_n))+4\\ & =2s+\ind-match_{\{K_2, C_5\}}(G_n)-2\leq {\rm reg}(I(G_n)^s)\leq 2s+\ind-match_{\{K_2, C_5\}}(G_n)-1\\ & \leq 2s+2n+6-1 < 2s+\frac{5n}{2}-1\leq 2s+{\rm cochord}(G_n)-1.
\end{align*}
Here, the third inequality follows from Theorem \ref{low}, the fourth inequality follows from \cite[Theorem 5.5]{js}, and the sixth inequality follows from the fact that $n\geq 13$.

Therefore, we proved the following result.

\begin{prop}
Using the notations as above, for every integer $s\geq 1$ and every integer $n\geq 13$, we have$$2s+\ind-match(G_n)-1< {\rm reg}(I(G_n)^s)< 2s+{\rm cochord}(G_n)-1.$$
\end{prop}





\end{document}